%% file: main.tex
\documentclass[paper=letter,11pt]{article}

\usepackage[leqno]{amsmath}

\makeatletter
\newcommand{\leqnomode}{\tagsleft@true}
\newcommand{\reqnomode}{\tagsleft@false}
\makeatother

\usepackage{amsfonts,amssymb,stackrel,amsthm,bbm,amssymb,array,enumitem,makeidx,xcolor,mathtools,xpatch,multirow,graphicx,svg}
\usepackage{euflag}
\usepackage[inline]{asymptote}
\usepackage[style=numeric, maxbibnames=99]{biblatex}
\usepackage{float}
\usepackage{comment}
\usepackage[hidelinks]{hyperref}
\usepackage{marvosym}

\newcounter{tbox}

\usepackage[margin=1in]{geometry}
\newtheorem{lemma}{Lemma}[section]
\newtheorem{definition}[lemma]{Definition}

\newtheorem{theorem}[lemma]{Theorem}
\newtheorem{observation}[lemma]{Observation}
\theoremstyle{definition}

\def\nospace{\kern -20pt}

\title{Orientations of graphs with at most one directed path between every pair of vertices}

\def\mail#1{\href{mailto:#1}{#1}}
\author{
    Barbora Dohnalová  \thanks{Department of Applied Mathematics, Charles University, Prague, Czechia, Email: \mail{bdohnalova@matfyz.cz, jirikalvoda@kam.mff.cuni.cz, gaurav@kam.mff.cuni.cz}. Supported by European Union’s Horizon 2020 research and innovation programme under the Marie Sklodowska-Curie grant agreement No. 823748.}
    \and
    Jiří Kalvoda \footnotemark[1]
    \and
    Gaurav Kucheriya \footnotemark[1] \thanks{Supported by GA\v{C}R grant 22-19073S, SVV–2023–260699.}
    \and
    Sophie Spirkl \thanks{Department of Combinatorics and Optimization, University of Waterloo, Canada Email: \mail{sspirkl@uwaterloo.ca}. We acknowledge the support of the Natural Sciences and Engineering Research Council of Canada (NSERC), [funding reference number RGPIN-2020-03912]. Cette recherche a \'et\'e financ\'ee par le Conseil de recherches en sciences naturelles et en g\'enie du Canada (CRSNG), [num\'ero de r\'ef\'erence RGPIN-2020-03912]. This project was funded in part by the Government of Ontario.}%
}

\date{}

\begin{document}
	\maketitle
	\begin{abstract}
	    Given a graph $G$, we say that an orientation $D$ of $G$ is a \emph{KT orientation} if, for all $u, v \in V(D)$, there is at most one directed path (in any direction) between $u$ and $v$. Graphs that admit such orientations have been used by Kierstead and Trotter (1992), Carbonero, Hompe, Moore, and Spirkl (2023), Bria\'nski, Davies, and Walczak (2024), and Gir\~ao, Illingworth, Powierski, Savery, Scott, Tamitegami, and Tan (2024) to construct graphs with large chromatic number and small clique number that served as counterexamples to various conjectures. 
     
        Motivated by this, we consider which graphs admit KT orientations (named after Kierstead and Trotter). In particular, we construct a graph family with small independence number (sub-linear in the number of vertices) which admits a KT orientation. We show that the problem of determining whether a given graph admits a KT orientation is NP-complete, even if we restrict ourselves to planar graphs. Finally, we provide an algorithm to decide if a graph with maximum degree at most 3 admits a KT orientation, whereas, for graphs with maximum degree 4, the problem remains NP-complete. 
	\end{abstract}

    \section{Introduction}

    All graphs in this paper are simple and finite. Let $G$ be a graph and $k \in \mathbb{N}$. A \emph{$k$-colouring} of $G$ is a function $f: V(G) \rightarrow \{1, \dots, k\}$ such that $f(u) \neq f(v)$ for all $uv \in E(G)$. If a $k$-colouring of $G$ exists, we say that $G$ is $k$-colourable. For a graph $G$, let $\chi(G)$ denote its \emph{chromatic number}, the minimum number of colors required so that $G$ is $\chi(G)$-colorable, and let $\omega(G)$ denote its \emph{clique number}. A family $\mathcal G$ of graphs is \textit{$\chi$-bounded} (see \cite{gyarfas1985problems}) if there exists a function $f$ such that $\chi(H)\le f(\omega(H))$ for all $G \in \mathcal G$ and each induced subgraph $H$ of $G$. We refer the reader to \cite{chi_boundedness} for a survey of $\chi$-boundedness. 

    Several counterexamples related to $\chi$-boundedness are based on a similar idea: We start with a directed graph $D$ with the property that for all $u, v \in V(D)$, there is at most one directed path with $u$ and $v$ as its ends (we call such a digraph a \emph{KT orientation}). Then, we fix $k \in \mathbb{N}$ and $S \subseteq \{1, \dots, k-1\}$; we construct $D(k, S)$ by adding an edge from $u$ to $v$ if there is a directed path from $u$ to $v$ in $D$, and its length is congruent, modulo $k$, to a number in $S$. The first example of this, due to Kierstead and Trotter \cite{colorful_ind_subg}, shows that graphs that admit an orientation with no induced monotone four-vertex path (that is, a $P_4$ with the orientation $\to\to\to$) are not $\chi$-bounded. Similar ideas were subsequently used \cite{brianski2024separating, trianglefree_contraexample, large_chromatic} to construct counterexamples to long-standing questions on $\chi$-boundedness. 

    This motivates the following definition: We say that a graph $G$ \textit{admits a KT orientation} if there is an orientation $D$ of $G$ which is a KT orientation. We formally define this in Section~\ref{sec:preliminaries} along with some preliminary observations and results.

    What can we say about graphs that admit KT orientations? Admitting a KT orientation suggests that the underlying graph is ``sparse'' in some sense; however, in Section \ref{sec:small_independence}, we show that there are graphs without linear-sized independent sets that admit a KT orientation using a variation of twincut graphs from \cite{twincut}. Next, we turn our attention to graphs that do not admit a KT orientation. Prior to this work, all known examples of such graphs contained cycles of length at most five \cite{sadhukhan2024shift}. Sadhukhan \cite{sadhukhan2024shift} conjectured that there are graphs with large girth that do not admit KT orientations. In Section \ref{sec:girth}, we settle this conjecture by constructing graphs of arbitrarily large girth with no KT orientation. 
    
    One of the applications of directed graphs that admit KT orientations relates to the Routing and Wavelength Assignment (RWA) problem in optical networks (see~\cite{mukherjee1997optical}). It aims to find routes and their associated wavelengths to satisfy a set of traffic requests while minimizing the number of wavelengths used. This problem is NP-hard in general, so it is usually split into two distinct problems: first step is to find routes for the requests and then for the second step, take these routes as input to solve the wavelength assignment problem. In \cite{bermond2013directed}, Bermond, Cosnard, and P\'erennes study the RWA problem for digraphs that are KT orientations; in this case, the routing is forced and thus the only problem to solve is the wavelength assignment one.
    
    In \cite{bermond2013directed}, Bermond, Cosnard, and P\'erennes  introduced the notion of good edge-labeling of a graph. A \emph{labeling} of a graph $G=(V,E)$ is a function $\phi:E \to \mathbb{R}$. The labeling is \emph{good} if the labels are distinct and for any ordered pair of vertices $(x,y)$, there is at most one path from $x$ to $y$ with increasing labels. In \cite{araujo2012good}, Araujo, Cohen, Giroire, and Havet show that deciding whether a graph has a good edge-labeling is NP-hard. Since a admitting good edge-labeling is similar admitting a KT orientation, \cite{bermond2013directed} conjectured that deciding whether a given graph has a KT orientation is NP-hard. In Section \ref{sec:np}, we prove this result. Furthermore, we also show that this problem is NP-complete, even when restricted to the following graph families: planar graphs, graphs with maximum degree at most four, and graphs of large girth. On the other hand, Section \ref{sec:algorithm} contains a polynomial-time algorithm for determining whether a graph with maximum degree at most 3 admits a KT orientation. 
    
    \section{Preliminaries}
    \label{sec:preliminaries}
    We begin with some definitions. 
    
    \begin{definition}
        Let $G$ be a graph and let $K$ be an orientation on $G$. We say that $K$ is a \emph{KT orientation} on $G$ if for all vertices $u,v$ of $G$, there is at most one directed path (in any direction) between $u$ and $v$. We say that $G$ \emph{admits a KT orientation} if there is at least one orientation of $G$ which is a KT orientation.
    \end{definition}
    In particular, we note that a KT orientation is necessarily acyclic. 
    
    \begin{figure}[H]
        \centering
        \label{fig:kt}
        \includegraphics[]{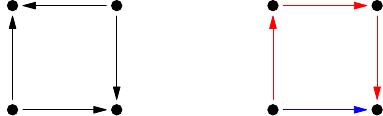}
        \caption{Left: A KT orientation of $C_4$. Right: Not a KT orientation of $C_4$ as there are two directed paths (in red and blue, respectively) between a pair of vertices.}
    \end{figure}

    \begin{observation}\label{obs:bip}
        Every $k$-colourable graph of girth at least $2k-1$ admits a KT orientation. In particular, every bipartite graph admits a KT orientation. 
    \end{observation}
    \nospace
    \begin{proof}
        Let $G$ be a graph with a $k$-colouring $f : V(G) \rightarrow \{1, \dots, k\}$. For every edge $uv \in E(G)$, we direct $uv$ as $u \rightarrow v$ if and only if $f(u) < f(v)$. Let $D$ be the resulting directed graph. 

        Suppose for a contradiction that $D$ is not a KT orientation. Since $f$ is strictly increasing along every directed path, it follows that $D$ is acyclic and every directed path in $D$ contains at most $k-1$ edges. Therefore, if there are vertices $u$ and $v$ in $D$ such that $D$ contains two distinct directed paths from $u$ to $v$, say $P_1$ and $P_2$, then $G$ contains an undirected cycle (using edges of $P_1$ and $P_2$) of length at most $2k-2$, a contradiction. 
    \end{proof}

    \begin{figure}[H]
        \centering
        \label{fig:bipartite}
        \includegraphics[]{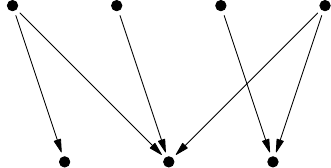}
        \caption{A bipartite graph with a KT orientation.}
    \end{figure}

    For an orientation $K$, we denote by $K^{\rm R}$ the orientation obtained by reversing the direction of each edge of $K$. Next, we observe that the property of admitting a KT orientation is closed under this operation of reversing the orientation.
    \begin{observation}\label{obs:rev}
        Let $G$ be a graph that admits a KT orientation $K$. Then $K^{\rm R}$ is also a KT orientation of $G$.
    \end{observation}
    It is easy to see that triangles do not admit a KT orientation; therefore, graphs which do admit a KT orientation are triangle-free. 

    Finally, we prove the following, which simplifies later proofs: 
    \begin{lemma} \label{lem:disjoint}
        If $D$ is a directed graph that is not a KT orientation, then there exist distinct $u, v \in V(D)$ such that there is an induced subgraph of $D$ which is the   union of two internally-disjoint directed paths between $u$ and $v$.
    \end{lemma} 
    \begin{proof}  
        Since $D$ is not a KT orientation, it follows that there exist distinct vertices $u, v \in V(D)$ such that $D$ contains two distinct directed paths $P$ and $Q$ between $u$ and $v$. Let us choose $P, Q, u, v$ with this property, and subject to that, with the sum of lengths of $P$ and $Q$ as small as possible. 

        Now, if there is a vertex $z$ which occurs both in the interior of $P$ and the interior of $Q$, then $D$ contains two distinct paths either between $z$ and $u$ or between $z$ and $v$ by taking the subpaths of $P$ and $Q$ from $z$ to one of $u$ or $v$. But this contradicts the choice of $P$ and $Q$. It follows that $P$ and $Q$ are internally disjoint, as desired. 

        If one of $P$ and $Q$ is not induced, then there is an edge $xy$ such that $x, y \in V(P)$ but $xy \not\in E(P)$ and $xy \neq Q$. Now $P, xy, x, y$ is a better choice than $P, Q, u, v$. Likewise, supposing that $x'y'$ is an edge between $V(P) \setminus \{u, v\}$ and $V(Q) \setminus \{u, v\}$, we can find a better choice of paths within either the union of the subpaths of $P$ and $Q$ from $u$ to $x'$ and $u$ to $y'$, respectively, or the union of the subpaths of $P$ and $Q$ from $v$ to $x'$ and $v$ to $y'$, respectively. This concludes the proof. 
    \end{proof}

    Lemma \ref{lem:disjoint} has the following easy consequence: 
    \begin{observation} \label{obs:cc}
        Let $D, D'$ be KT orientations of graphs $G, G'$ with disjoint vertex sets. Let $x \rightarrow y$ be an edge of $D$, and let $x' \rightarrow y'$ be an edge of $D'$. Then, taking the union of $D$ and $D'$ and identifying $x$ with $x'$ as well as $y$ with $y'$ results in a KT orientation. 
    \end{observation}
    
    \section{Graphs that admit KT orientations with no linear-size independent set}
    \label{sec:small_independence}
    \input copycut.tex
    \input{girth.tex}

    \input np.tex

    \input{three_v2/three_v2.tex}

    \section{Acknowledgement}
        We would like to thank the organisers of the DIMACS REU 2023 for providing the wonderful opportunity to work on this project during the summer. We would also like to express our gratitude to G\'abor Tardos for providing invaluable feedback to improve the overall presentation. 
         
\printbibliography
    
\end{document}

%% file: copycut.tex
It was shown in \cite{colorful_ind_subg} that Zykov graphs admit KT orientations. These graphs have linear-size independent sets and unbounded fractional chromatic number \autocite{fractional_chromatic_zykov}. So a natural question to ask is if there are graphs with small independence number that admit KT orientations.

A natural first attempt towards proving this would be the following: It is well-known that by replacing each vertex $v$ by a number of non-adjacent copies of $v$ (that is, vertices with the same neighborhood as $v$) proportional to the dual solution for the fractional chromatic number linear program, we obtain a graph $G$ whose stability number equals $|V(G)|/f$, where $f$ is the fractional chromatic number of $G$. 

However, admitting a KT orientation is not closed under taking non-adjacent copies of a vertex (for example, take two non-adjacent copies of every vertex in a five-cycle), and therefore, we cannot leverage fractional chromatic number directly. 

In this section, we create a sequence of graphs \(G_1, G_2, \dots\), each of which admits a KT orientation, but their largest independent set is sub-linear and thus their fractional chromatic number is unbounded. 

\begin{definition} \label{def:f}
Let us define a sequence $F_1, F_2, \dots$ as follows. 
\begin{align*}
    F_1 &= 1 \\
 F_{k+1} &= F_k + \frac{1}{F_k} \hskip 2mm \hbox{for all $k \le 1.$} 
\end{align*}
\end{definition}

Writing $\alpha(G)$ for the size of a largest independent set in $G$, we will show that $\alpha(G_k) = |V(G_k)|/F_k$ for all $k \in \mathbb{N}$. 

We note that the same recursive expression was used in \autocite{fractional_chromatic_zykov, larsen1995fractional}, where it was shown to be the fractional chromatic number of Zykov graphs and Mycielski graphs, respectively. In particular, as noted in \cite{larsen1995fractional}, it is well known that $F_k/\sqrt{2k} \rightarrow 1$ as $k \rightarrow \infty$. 

\hypertarget{construction}{%
\subsection{Construction}\label{construction}}

Since our construction is based on the construction of twincut graphs \autocite{twincut}, we begin by describing the construction of twincut graphs. Let us denote the \(i\)-th twincut graph as \(G_i'\) and its number of vertices by \(n_k'\). In each twincut graph, some of its vertices are called \emph{branch vertices}, and we refer to the remaining vertices as \emph{inner vertices}. The set of branch vertices forms an independent set. 

Let \(G_1'\) be a one-vertex graph (a branch vertex). We will construct \(G_{k+1}'\) from \(G_k'\) (for \(k\ge 1\)), as follows (see Figure \ref{fig:twincut}). Fix an enumeration of $V(G_k')$ as $x_1, \dots, x_{n_k'}$. First, we take the original $G_k'$ and add to it a new copy of \(G_k'\) for each branch vertex $v$ of $G_k'$, making sure that the original $G_k'$ and all its copies are pairwise disjoint. Denote this graph by \(G_k'(v)\). Then we replace each branch vertex $v$ of $G_k'$ by \(n_k'\) copies $v_1, \dots, v_{n_k'}$ of $v$. Finally, for each branch vertex $v$ of $G_k'$ and every $i \in \{1, \dots, n_k'\}$, we add an edge between $v_i$ and the vertex corresponding to $x_i$ in the copy $G_k'(v)$ of $G_k'$. The resulting graph is the graph $G_{k+1}'$. The set $\{v_i: v \textnormal{ is a branch vertex of } G_k', i \in \{1, \dots, n_{k'}\}\}$ is the set of branch vertices of $G_{k+1}'$, while the remaining vertices are inner vertices.  

    \begin{figure}[H]
        \centering
        \includegraphics[]{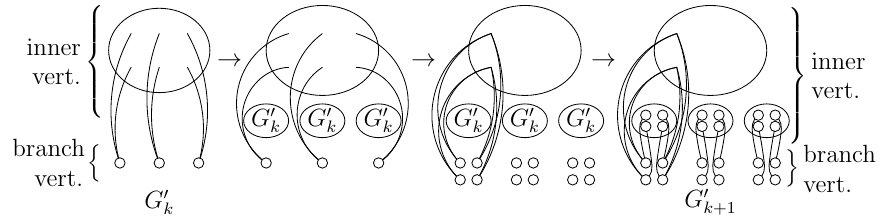}
        \caption{Construction of twincut graphs.}
        \label{fig:twincut}
    \end{figure}

Let $d = (d_1, \dots)$ be a sequence of non-negative integers. Let us now describe the construction of \emph{$d$-copycut graphs} $G_1, G_2, \dots $. Again, the vertices of each graph $G_k$ are partitioned into \emph{branch vertices} and \emph{inner vertices}. Furthermore, $G_k$ only depends on $d_1, \dots, d_{k-1}$.  

As before, \(G_1\) is a one-vertex graph (a branch vertex). For $k \geq 1$, we construct $G_{k+1}$ as follows (see Figure \ref{fig:copycut}). We denote the number of vertices of $G_k$ by  \(n_k\). Let $b_1, \dots, b_l$ be an enumeration of the branch vertices of $G_k$. We start by taking $d_k$ pairwise disjoint copies of $G_k$, say $G_k^1, \dots, G_k^{d_k}$. Then, for each $r \in \{1, \dots, l\}$, we identify the $d_k$ copies of the vertex $b_r$ in $G_k^1, \dots, G_k^{d_k}$; denote the new vertex obtained by this identification as $c_r$. Let $G_k^*$ denote the resulting graph, and let us call $c_1, \dots, c_l$ its branch vertices. All the remaining vertices are inner vertices. 

Now, we proceed as in the definition of twincut graphs. Fix an enumeration of $V(G_k)$ as $x_1, \dots, x_{n_k}$. First, we take the original $G_k$ and add to it a new copy of \(G_k\) for each branch vertex $v$ of $G_k^*$. Denote this graph by \(G_k'(v)\). Then we duplicate each branch vertex $v$ of $G_k^*$ into \(n_k\) new vertices $v_1, \dots, v_{n_k}$. Finally, for each branch vertex $v$ of $G_k^*$ and every $i \in \{1, \dots, n_k\}$, we add an edge between $v_i$ and the vertex $x_i$ in the copy $G_k(v)$ of $G_k$. The resulting graph is $G_{k+1}$. The set $\{v_i: v \textnormal{ is a branch vertex of } G_k^*, i \in \{1, \dots, n_{k}\}\}$ is the set of branch vertices of $G_{k+1}$, while the remaining vertices are inner vertices. 

Note that the construction is similar to blowing up non-branch vertices of $G_k$ into $d_i$ copies, but not exactly, and this small distinction ensures that the resulting graph admits a KT-orientation.

    \begin{figure}[H]
        \centering
        \includegraphics[]{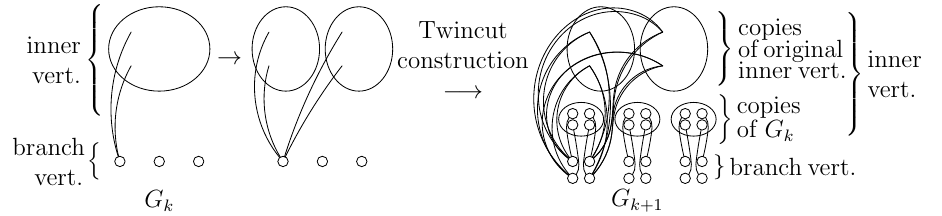}
        \caption{Construction of copycut graphs.} 
        \label{fig:copycut}
    \end{figure}

\begin{theorem}
    Let $d = (d_1, \dots)$ be a sequence of non-negative integers. For all $k \in \mathbb{N}$, the $d$-copycut graph $G_k$ admits a KT orientation. 
\end{theorem}
\begin{proof}
    We will prove by induction on $k$ that every $d$-copycut graph $G_k$ admits a KT orientation such that each branch vertex is a source (see Definition \ref{def:source}). Note that $G_1, G_2$ admit such KT orientations. 

    Let us now suppose that $G_k$ admits a KT orientation, say $D$, such that all branch vertices are sources. Then, we define an orientation on $G_k^*$ by orienting $G_k^1, \dots, G_k^{d_k}$ according to $D$, and preserving the orientation when identifying vertices. Next, for each branch vertex $v$ of $G_k^*$, we orient $G_k(v)$ according to $D$, and when replacing $v$ by copies $v_1, \dots, v_{n_k}$, we orient their incident edges according to the orientation of the corresponding edges of $v$; since by our assumption, $v$ is a source, it follows that $v_1, \dots, v_{n_k}$ are sources. Finally, when adding edges between $x_i$ and $v_i$, we orient the edge as $v_i \rightarrow x_i$, preserving that $v_i$ is a source. Let us call the resulting orientation $D'$. 

    From this construction, it is immediate that all branch vertices of $G_{k+1}$ are sources of $D'$. It remains to show that $D'$ is a KT orientation. 

    Suppose for a contradiction that this is not the case. By Lemma \ref{lem:disjoint}, we may choose distinct $s, t \in V(G_{k+1})$ such that there are at least two distinct and internally disjoint paths $P, P'$ between $s$ and $t$. Since all branch vertices are sources, it follows that $V(P) \cup V(P')$ contains at most one branch vertex, namely $s$ or $t$. If neither $s$ nor $t$ is a branch vertex of $G_{k+1}$, then $P$ and $P'$ are both contained in a copy of $G_k$ oriented according to $D$, a contradiction as $D$ is a KT orientation of $G_k$. 

    So we may assume that $S$ is a branch vertex of $G_{k+1}$. There are two possibilities for $t$: 
    \begin{itemize}
        \item If $t \in G_k^i$ for some $i \in \{1, \dots, d_k\}$, then $P$ and $P'$ are both contained in $G_k^i$, which is oriented according to $D$, a contradiction. 
        \item If $t \in G_k(v)$ for some branch vertex $v$ of $G^*$, then all vertices of $P, P'$ except for $s$ are contained in $G_k(v)$. Moreover, since each branch vertex of $G_{k+1}$ has at most one neighbour in $G_k(v)$, it follows that $s = v_i$ for some $i \in \{1, \dots, n_k\}$. But as $v_i$ has a unique neighbour in $G_k(v)$, namely $x_i$, it follows that both $P$ and $P'$ contain $x_i$, and so $x_i = t$ and $P = P'$, a contradiction.  
    \end{itemize}
    This proves that $D'$ is a KT orientation of $G_{k+1}$, as desired. 
\end{proof}

\begin{theorem}
    There exists a sequence $d = (d_1, \dots)$ such that for all $k \in \mathbb{N}$, we have that $\alpha(G_k) = |V(G_k)|/F_k$, where $F_k$ is as in Definition \ref{def:f}. 
\end{theorem}
\begin{proof}
    We define $d_1, d_2, \dots$ iteratively, starting with $d_1 = 1$. 

    Now we may assume that $d_1, \dots, d_{k-1}$ are defined, and therefore, $G_1, \dots, G_k$ are defined. We let $d_k = |V(G_k)| - \alpha(G_k)$. 

    We will prove by induction that the statement of the theorem holds, and moreover, that the set of branch vertices is a maximum independent set in each $G_k$. This is true for $k = 1$ and $k = 2$ since $F_1 = 1$ and $F_2 = 2$.  
    
    Now let us assume that $k \geq 2$, and the statement holds for $k$; we would like to show that it holds for $k+1$. Let us write $\alpha_k = \alpha(G_k)$ and recall that $n_k = |V(G_k)|$. From the induction hypothesis, it follows that $\alpha_k$ is the number of branch vertices of $G_k$. 

    We compute: 
    \begin{align*}
        n_{k+1} &= d_k(n_k - \alpha_k) + 2 n_k\alpha_k\\
        &= (n_k - \alpha_k)^2 + 2 n_k\alpha_k\\
        &= n_k^2 + \alpha_k^2. 
    \end{align*}
    The number of branch vertices of $G_{k+1}$ is $n_k\alpha_k$, and we have: 
    \begin{align*}
        n_k\alpha_k F_{k+1} &=n_k\alpha_k \cdot(F_k + 1/F_k) \\
        &= n_k\alpha_k \cdot (n_k/\alpha_k) + n_k\alpha_k \cdot (\alpha_k/n_k) \\
        &= n_k^2 + \alpha_k^2.        
    \end{align*}
    Therefore, it suffices to prove that $H_0$, the set of branch vertices of $G_{k+1}$, is a maximum independent set of $G_{k+1}$. 

    Consider an independent set $H$ in $G_{k+1}$. Partition $V(G_{k+1})$ into $d_k+l$ sets: the first $d_k$ sets are $A_j$ consisting of the non-branch vertices of $G_k^j$ for each $j$, the last $l$ are the sets $B_v$ consisting of the copies $v_i$ of a branch vertex $v$ of $G_k^*$ together with the vertices of $G_k(v)$.

    Firstly observe that $(H\cap A_j)\cup H_0$ must be independent in the copy $G_k^j$ of $G_k$, therefore $|H\cap A_j|\le\alpha(G_k)-|H_0|$. For $v\notin H_0$, $H\cap B_v$ is an independent set in $G_k(v)$ (another copy of $G_k$), so its size is at most $\alpha(G_k)$. Finally, for $v\in H_0$ we have $|H\cap B_v|\le n_k$ as $B_v$ induces a subgraph with $2n_k$ vertices containing a perfect matching. Here $l=\alpha(G_k)$. Summing all this, we have $|H|\le d_k(l-|H_0|)+(l-|H_0|)l+|H_0|n_k=n_kl$, which is exactly the number of branch vertices in $G_{k+1}$.
\end{proof}

%% file: girth.tex
\section{Graphs with large girth and no KT orientation} \label{sec:girth}

In this section, we show that there are graphs with large girth and with no KT orientation. 

We adapt a method of Duffus, Ginn, and R\"odl \cite{duffus1995computational}. For this, we need the following definitions. 

\begin{definition}
    Fix $ k \in \mathbb{N}$. An \emph{oriented $k$-uniform hypergraph} $H$ consists of a finite set $V(H)$ of vertices, and a set $E(H) \subseteq V(H)^k$ of edges, each of which is an ordered $k$-tuple of distinct vertices. 
    
    A \emph{cycle} in a hypergraph is a sequence of pairwise distinct vertices and edges $v_1, e_1, v_2, e_2, \dots, v_l, e_l$ such that for all $i \in \{1, \dots, l-1\}$, the edge $e_i$ contains $v_i$ and $v_{i+1}$, and $e_l$ contains $v_l$ and $v_1$. 
    
    Given a cycle in a hypergraph, its \emph{length} is its number of vertices. The girth of a hypergraph is the length of a shortest cycle in it, or $\infty$ if no cycle exists. 
\end{definition}

\begin{theorem}[Duffus, Ginn, R\"odl \cite{duffus1995computational}, Lemma 3.3] \label{thm:dgr}
    Let $l, k \geq 2$ be integers, and let $\delta > 0$. Let $H$ be an oriented $k$-uniform hypergraph. Then there is an oriented $k$-uniform hypergraph $H'$ with vertex set $V(H) \times X$ (for some set $X$) such that: 
    \begin{itemize}
        \item $H'$ contains no cycle of length at most $l$; 
        \item Writing $X_v = \{(v, x) : x \in X\} \subseteq V(H) \times X$, we have that for every edge $e = (v_1, \dots, v_k) \in E(H)$, and for all sets $X_{v_i}' \subseteq X_{v_i}$ with $|X_{v_i}'| \geq \delta |X_{v_i}|$ for all $i \in \{1, \dots, k\}$, there exist $x_1, \dots, x_k$ with $x_i \in X_{v_i}' $ such that $(x_1, \dots, x_k) \in E(H')$. 
    \end{itemize}
    
\end{theorem}

\begin{theorem} \label{thm:girth}
    Let $k \in \mathbb{N}$. Then there is a graph with girth at least $k$ which does not admit a KT orientation. 
\end{theorem}
\begin{proof}
Let $H$ be the following oriented $k$-uniform hypergraph: 
\begin{itemize}
    \item $V(H) = \{v_1, \dots, v_k\}$; 
    \item $E(H) = \{(v_{\pi(1)}, \dots, v_{\pi(k)}): \pi \in S_k\}$, where $S_k$ is the set of all permutations of $\{1, \dots, k\}$. 
\end{itemize}
In other words, $H$ is the complete oriented $k$-uniform hypergraph: every possible edge is present. 

Let $H'$ and $X$ be as promised by Theorem \ref{thm:dgr} applied to $H$ with $l=k$ and $\delta = 1/k$. 

Let us define $G$ as follows. We have $V(G) = V(H') = V(H) \times X$. For every edge $(x_1, \dots, x_k)$ of $H'$,  we add the following edges to $G$: 
\begin{itemize}
    \item $x_jx_{j+1}$ for all $j \in \{1, \dots, k-1\}$; and 
    \item $x_1x_k$. 
\end{itemize}

We claim that $G$ has girth $k$. (This follows from arguments in \cite{duffus1995computational}, but was not stated as a theorem.) Suppose not; and let $C$ be a cycle in $G$ of length less than $k$. Then not all edges of $C$ were added due to the same edge $e$ of $H'$. But now we can use $C$ to obtain a cycle of length at most $k$ in $H'$, a contradiction. 

Now suppose for a contradiction that $G$ admits a KT orientation $D$. Since every KT orientation is acyclic, it follows that there is a topological ordering $q_1, \dots, q_m$ of $V(G)$ such that for every edge $q_iq_j$ of $G$ with $i < j$, we have $q_i \rightarrow q_j$ in $D$. 

For $i \in \{1, \dots, k\}$, we let $X_i = \{v_i\} \times X = \{(v_i, x) : x \in X\}$. 

The following is again implicit in \cite{duffus1995computational}, where it is shown that in every ordering of the vertices of $H'$, there is some edge which appears in order.

We define a permutation $\pi$ and sets $X_i'$ for $i \in \{0, \dots, k-1\}$ by defining $\pi(1), \pi(2), \dots$ iteratively as follows: We define $j_0 = 0$. We assume that $\pi(1), \dots, \pi(i) \in \{1, \dots, k\}$ are defined and distinct. We also assume that $j_{i}$ is defined, and that for each $l \in \{1, \dots, k\} \setminus \{\pi(1), \dots, \pi(i)\}$, we have $|\{q_{j_i+1}, \dots, q_m\} \cap X_l| \geq (k-i)|X|/k.$ This conditions holds for $i = 0$ since $|X_l| = |X|$ for all $l \in \{1, \dots, k\}$. 

Now, let $i \in \{0, \dots, k-1\}$ and let $S_i = \{\pi(1), \dots, \pi(i)\}$. Let $j_{i+1}$ be minimum such that: 
\begin{itemize}
    \item $j_{i+1} \geq j_i + 1$; and
    \item there exists $l \in \{1, \dots, k\} \setminus S_i$ such that $|X_l \cap \{q_{j_i+1}, \dots, q_{j_{i+1}}\}| \geq |X|/k$. 
\end{itemize}
Note that $j_{i+1}$ is well-defined since $|S_i| = i$, and so $\{1 \dots, k\} \setminus S_i \neq \emptyset$, and $|X_l \cap \{q_{j_i+1}, \dots, q_m\}| \geq (k-i)|X|/k \geq (k-(k-1))|X|/k \geq |X|/k$. 

Letting $l$ be as in the definition of $j_{i+1}$, we set $\pi(i+1) = l$. We let $X_{l}' = |X_l \cap \{q_{j_i+1}, \dots, q_{j_{i+1}}\}|$. Then, since $\pi(i+1) \not\in S_i$, it follows that $\pi(1), \dots, \pi(i+1) \in \{1, \dots, k\}$ are defined and distinct. Moreover, $j_{i+1}$ is defined. In addition, for each $l' \in \{1, \dots, k\} \setminus (S_i \cup \{l\})$, we have that $|X_{l'} \cap \{q_{j_i+1}, \dots, q_{j_{i+1}}\}| \leq |X|/k$, and so 
\begin{align*}
    |\{q_{j_{i+1}+1}, \dots, q_m\} \cap X_{l'}| &=  |\{q_{j_i+1}, \dots, q_m\} \cap X_{l'}| - |X_{l'} \cap \{q_{j_i+1}, \dots, q_{j_{i+1}}\}|\\
    &\geq (k-i)|X|/k - |X|/k = (k-i-1)|X|/k,
\end{align*}
as desired. 

We have defined sets $X_i'$ for all $i \in \{1, \dots, k\}$ with the following properties: 
\begin{itemize}
    \item For all $i \in \{1, \dots, k\}$, $X_i' \subseteq X_i$ and $|X_i'| \geq |X|/k = |X_i|/k$;
    \item If $q_r \in X'_{\pi(i)}$ and $q_s \in X'_{\pi(j)}$ with $i < j$, then $r < s$. 
\end{itemize}
In other words, in our topological ordering of $D$, the sets $X'_{\pi(1)}, X'_{\pi(2)}, \dots, X'_{\pi(k)}$ appear in this order. 

Now, applying Theorem \ref{thm:dgr}, and using that $(v_{\pi(1)}, \dots, v_{\pi(k)}) \in E(H)$, it follows that there is an edge $(x_1, \dots, x_k) \in E(H')$ such that $x_i \in X'_{\pi(i)}$ for all $i \in \{1, \dots, k\}$. But then, in the topological ordering of $D$, the vertices $x_1, \dots, x_k$ appear in this order; and so, from our construction of $G$ and the definition of a topological ordering, we obtain that $x_1 \rightarrow x_k$ and $x_1 \rightarrow x_2 \rightarrow \dots \rightarrow x_k$ are two distinct directed paths from $x_1$ to $x_k$ in $D$, a contradiction. This proves that $D$ is not a KT orientation, which concludes the proof. 
\end{proof}

A similar argument to the proof of Theorem \ref{thm:girth} establishes the stronger statement that for every vertex-ordered graph $J$, we can construct a graph $G$ with the same girth as $J$ such that every vertex-ordering of $G$ contains an ordered induced copy of $J$.

%% file: np.tex
\section{NP-Completeness Results}
\label{sec:np}

In this section, we show that it is NP-complete to decide if a given graph admits a KT orientation. We first give a reduction for general graphs, and then show how to adapt this construction to certain restricted graph classes. 

\begin{lemma} \label{lem:ktalg}
    There is a polynomial-time algorithm which decides, for a given digraph $D$, whether $D$ is a KT orientation or not. 
\end{lemma}
\begin{proof}
    We first check whether $D$ is acyclic, and output that $D$ is not a KT orientation if $D$ contains a directed cycle. Now we may assume that $D$ is acyclic; and we may compute the number of directed paths between every pair of vertices (for example, via dynamic programming using a topological ordering). 
\end{proof}
In particular, it follows that the problem of deciding whether a given graph admits a KT orientation is in NP. 

\subsection{NP-Hardness in Arbitrary Graphs}

    In this subsection, we will prove Theorem \ref{thm:nphard1} (below) using a reduction from the monotone not-all-equal 3-SAT problem \cite{nae3sat} (also known as the 2-coloring of 3-uniform hypergraphs problem):
    \begin{definition}
        \emph{Monotone not-all-equal 3-SAT} is the following problem:
        \begin{itemize}
            \item Input: Variables $x_1, \dots, x_n$ and clauses $C_1, \dots, C_m$; each clause consists of exactly three variables (no negations). 
            \item Output: Yes if there is a truth assignment such that each variable $x_1, \dots, x_n$ is set to either True or False, with the property that each clause contains at least one variable which is set to True and one variable which is set to False; and No otherwise. 
        \end{itemize}
    \end{definition}

    \begin{theorem}[\cite{nae3sat}]
        The monotone not-all-equal 3-SAT problem is NP-complete. 
    \end{theorem}

    We will use this to show: 

    \begin{theorem} \label{thm:nphard1}
        Deciding whether a given graph admits a KT orientation is NP-hard.
    \end{theorem}

    Before proving Theorem \ref{thm:nphard1}, we start with some definitions and lemmas. 

    \begin{definition} \label{def:kladder}
        Let $k \in \mathbb{N}$. A \emph{$k$-ladder} is a graph with vertex set $\{v_1, \dots, v_k, w_1, \dots, w_k\}$ and edge set $$\{v_iw_i : i \in \{1, \dots, k\}\} \cup \{v_iv_{i+1} : i \in \{1, \dots, k-1\}\} \cup  \{w_iw_{i+1} : i \in \{1, \dots, k-1\}\}.$$ In other words, the $k$-ladder graph is the $2 \times k$ grid.
    \end{definition}
    
  \begin{figure}[H]
        \centering
        \includegraphics[]{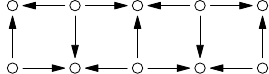}
        \label{fig:copytape}
        \caption{A KT orientation of a $5$-ladder.}
    \end{figure}

    \begin{definition}\label{def:source}
         A \emph{source} in a digraph is a vertex with no in-neighbours; a \emph{sink} in a digraph is a vertex with no out-neighbours. 
    \end{definition}   
    \begin{lemma}\label{lem:chain}
        There is exactly one KT orientation $D$ of a $k$-ladder such that $v_1 \rightarrow w_1$ in $D$. Furthermore, the orientation $D$ has the property that for $i \in \{1, \dots, k\}$, we have $v_i \rightarrow w_i$ in $D$ if and only if $i$ is odd. Also, each vertex of $D$ is a source or a sink. 
    \end{lemma}
    \begin{proof}
        We prove this by induction on $k$. The statement holds for $k=1$, as $v_1w_1$ is the only edge of a $1$-ladder. 
        
        Now, we first prove that every $k$-ladder has an orientation as desired, by taking the bipartition with $A = \{v_i : i \textnormal{ odd}\} \cup \{w_i : i \textnormal{ even}\}$ and $B = \{v_i : i \textnormal{ even}\} \cup \{w_i : i \textnormal{ odd}\}$. Then, as in Observation \ref{obs:bip}, we may obtain a KT orientation with the desired properties by directing every edge from $A$ to $B$. 
        
        It is easy to verify that there are exactly two KT orientations of a four-cycle, shown in Figure \ref{fig:4cycle}. 
         \begin{figure}[H]
        \centering
        \includegraphics[]{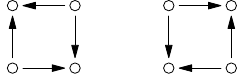}
        \caption{Possible orientations of a four-cycle.}
        \label{fig:4cycle}
    \end{figure}
     Now consider a KT orientation $D$ of a $k$-ladder such that $v_1 \rightarrow w_1$. Then, by induction, we may assume that $v_{k-1} \rightarrow w_{k-1}$ if $k-1$ is odd, and $w_{k-1} \rightarrow v_{k-1}$ if $k-1$ is even. Since $v_{k-1}$-$v_k$-$w_k$-$w_{k-1}$-$v_{k-1}$ is a four-cycle, and since there are only two KT orientations of a four-cycle as shown in Figure \ref{fig:4cycle}, it follows that $v_k \rightarrow w_k$ if and only if $k$ is odd. Note that this determines the orientations of all edges of the form $v_iv_{i+1}$ and $w_iw_{i+1}$.
    \end{proof}

    We use the ladder to propagate information; more precisely, if we have two edges $e = vw$ and $f = v'w'$ in a graph, we may create a $k$-ladder with $k$ odd and identify $v$ with $v_1$; $w$ with $w_1$; $v'$ with $v_k$; and $w'$ with $w_k$ to enforce that in every KT orientation, either both $v \rightarrow w$ and $v' \rightarrow w'$, or neither. 

    Let us now show how to create a clause gadget. 
    \begin{lemma} \label{lem:clause}
        Let $H$ be a five-cycle with vertices $v_1, \dots, v_5$ and edges $\{v_1v_2, v_2v_3, v_3v_4, v_4v_5, v_5v_1\}$. Then there is no KT orientation of $H$ such that $v_1 \rightarrow v_2$, and $v_2 \rightarrow v_3$, and $v_5 \rightarrow v_4$. Likewise, there is no KT orientation such that $v_2 \rightarrow v_1$, and $v_3 \rightarrow v_2$, and $v_4 \rightarrow v_5$. However, every other choice of orientation of the edges $v_1v_2, v_2v_3, v_4v_5$ extends to a KT orientation of $H$. 
    \end{lemma}
    \begin{proof}
        Let $D$ be an orientation of $H$. Firstly, observe that an orientation of a cycle is KT if and only if it has at least 4 vertices that are sources or sinks. By Observation \ref{obs:rev}, we may (by versing all edges if necessary) assume that $v_1 \rightarrow v_2$ in $D$. Now assume that the vertex $v_2$ is neither a source nor a sink, that is $v_2 \rightarrow v_3$. Therefore all the remaining vertices have to be either sources or sinks in the KT orientation of $H$. However if $v_5 \rightarrow v_4$, then one of $v_5$ and $v_1$ is also not a sink, and the first claim of the lemma follows. 

        For the second claim, we again assume that $v_1 \rightarrow v_2$; we extend the orientations as follows: 
        \begin{itemize}
            \item $v_2 \rightarrow v_3$ and $v_4 \rightarrow v_5$: Orient the remaining edges as $v_4 \rightarrow v_3$ and $v_1 \rightarrow v_5$. 
            \item $v_3 \rightarrow v_2$ and $v_4 \rightarrow v_5$: Orient the remaining edges as $v_1 \rightarrow v_5$ and $v_3 \rightarrow v_4$. 
            \item $v_3 \rightarrow v_2$ and $v_5 \rightarrow v_4$: Orient the remaining edges as $v_1 \rightarrow v_5$ and $v_3 \rightarrow v_4$.
        \end{itemize}
        Each of these is a KT orientation. \end{proof}
    
    The following lemma makes it easier to show that a given orientation is a KT orientation.  

    \begin{lemma} \label{lem:sourcesink}
        Let $D$ be a digraph, and let $u$ be a source or a sink in $D$. Suppose further that every neighbour of $u$ in $D$ is a source or a sink in $D$. Then $D$ is a KT orientation if and only if $D \setminus u$ is a KT orientation. 
    \end{lemma}

    \begin{proof}
        Since KT orientations are monotone under taking subgraphs, it follows that if $D$ is a KT orientation, then so is $D \setminus u$. 
        
        It remains to show the converse. By symmetry (using Observation \ref{obs:rev}), we may assume that $u$ is a source, and consequently, all neighbours of $u$ are sinks. Let $P$ be a directed path of $D$ which contains $u$. Then, since $u$ is a source, it follows that $u$ is the first vertex of $P$. Moreover, since every neighbour of $u$ is a sink, it follows that $P$ contains at most two vertices. 

        Now suppose that $D$ is not a KT orientation, and let $Q, Q'$ be two distinct directed paths between vertices $x, y \in V(D)$ which certify that $D$ is not a KT orientation. If $u$ is not a vertex of $Q, Q'$, then $D \setminus u$ is not a KT orientation as $Q, Q'$ are paths in $D \setminus u$. Therefore, we may assume that $x = u$. By the previous paragraph, it follows that $Q, Q'$ each contain at most two vertices, and therefore exactly two vertices, namely $x$ and $y$. But now $Q = Q'$, a contradiction. This shows that if $D$ is not a KT orientation, then $D \setminus u$ is not a KT orientation, finishing the proof. 
    \end{proof}

    \begin{proof}[Proof of Theorem \ref{thm:nphard1}]
        Let $x_1, \dots, x_n$ and $C_1, \dots, C_m$ be an instance of monotone not-all-equal 3-SAT. For each variable $x_i$, we start by taking two distinct vertices $y_i, z_i$ and add an edge $y_iz_i$. For each clause $C_j$, we add a copy of a five-cycle with vertex set $\{v_1^j, v_2^j, v_3^j, v_4^j, v_5^j\}$ and edge set $\{v_1^jv_2^j, v_2^jv_3^j, v_3^jv_4^j, v_4^jv_5^j, v_1^jv_5^j\}$. We also define $i^j_1, i^j_2, i^j_3$ such that $C_j$ contains variables $x_{i^j_1},x_{i^j_2},x_{i^j_3}$. 

        Then, for each clause $C_j = x_{i^j_1} \lor x_{i^j_2} \lor x_{i^j_3}$ and for each $l \in \{1, 2, 3\}$, we add a new 3-ladder with vertices $\{v_1,v_2,v_3,w_1,w_2,w_3\}$ and identify some of its vertices with the existing vertices as follows: 
            \begin{itemize}
                \item $v_1$ is identified with $y_{i_l^j}$; 
                \item $w_1$ is identified with $z_{i_l^j}$
                \item $v_3$ is identified with 
                $v_1^j$ if $l=1$, with $v_2^j$ if $l = 2$, and with $v_5^j$ if $l = 3$; 
                \item $w_3$ is identified with 
                $v_2^j$ if $l=1$, with $v_3^j$ if $l = 2$, and with $v_4^j$ if $l = 3$.  
            \end{itemize}

    Let us call the resulting graph $G$. Suppose first that $G$ admits a KT orientation $D$. Let us define a truth assignment as follows: For $i \in \{1, \dots, n\}$, if $y_i \rightarrow z_i$ in $D$, we set $x_i$ to True; otherwise, we set $x_i$ to False. Suppose that there is a clause $C_j$ with variables $x_{i_1^j}, x_{i_2^j}, x_{i_3^j}$ which are all set to True (the case in which they are all set to False is analogous). Then, using the 3-ladder gadgets and Lemma \ref{lem:chain}, we conclude that in $D$, we have $v_1^j \rightarrow v_2^j$ and $v_2^j \rightarrow v_3^j$, and $v_5^j \rightarrow v_4^j$. But now, by Lemma \ref{lem:clause}, the subgraph of $D$ induced by $\{v_1^j, \dots, v_5^j\}$ is not a KT orientation, a contradiction. This proves that if $G$ admits a KT orientation, then our instance of monotone not-all-equal 3-SAT is satisfiable. 

    It remains to show the converse, that is, we assume that a truth assignment of $x_1, \dots, x_n$ satisfying the monotone not-all-equal 3-SAT instance is given, and we aim to produce a KT orientation of $G$. 

    For $i \in \{1, \dots, n\}$, we direct $y_i \rightarrow z_i$ if $x_i$ is True, and $z_i \rightarrow y_i$ if $x_i$ is False. For each 3-ladder gadget whose edge $v_1w_1$ was identified with $y_iz_i$, we choose the unique orientation given by Lemma \ref{lem:chain} which extends the orientation of the edge $y_iz_i$ we chose. 

    For all $j \in \{1, \dots, m\}$, it follows that the edges $v_1^jv_2^j, v_2^jv_3^j$ and $v_4^jv_5^j$ have received an orientation already (via the 3-ladder gadgets) and, since not all variables in $C_j$ are assigned True, and not all are assigned False, it follows that this orientation can be extended to $v_3^jv_4^j$ and $v_1^jv_5^j$ using Lemma \ref{lem:clause} in such a way that the restriction of our orientation to $\{v_1^j, \dots, v_5^j\}$ is a KT orientation. 
    
    Let us denote the resulting orientation of $G$ as $D$. From the orientations of the 3-ladder gadgets as in Lemma \ref{lem:chain}, we conclude that for all $i \in \{1, \dots, n\}$, the vertices $y_i, z_i$ and all their neighbours are sources or sinks. Therefore, by Lemma \ref{lem:sourcesink}, it suffices to verify that $D' = D \setminus \{y_i, z_i : i \in \{1, \dots, n\}\}$ is a KT orientation.

    Each component $D''$ of $D'$ consists of 11 vertices, namely, for some $j \in \{1, \dots, m\}$, the vertices $v_1^j, \dots, v_5^j$, as well as two vertices from each of the three 3-ladder gadgets associated with $C_j$; all six of these vertices are sources or sinks in $D$ and hence in $D''$. Now, the restriction $D$ to the five-cycle on $v_1^j, \dots, v_5^j$ is a KT orientation, and the restriction of $D$ to each of its 3-ladder gadgets is a KT orientation. But $D''$ is obtained by identifying three edges of the five-cycle on $v_1^j, \dots, v_5^j$ with four-cycles contained in ladder gadgets; so by Observation \ref{obs:cc}, $D''$ is a KT orientation.
    \end{proof} 

We briefly mention the following: 
\begin{theorem}\label{thm:npharddeg4}
    Deciding if a given graph of maximum degree at most four admits a KT orientation is NP-complete. 
\end{theorem}
The proof is similar to the proof of Theorem \ref{thm:nphard1}, except that rather than identifying the edge $y_iz_i$ with several edges in different 3-ladder gadgets, we replace the edge $y_iz_i$ with a $k$-ladder for some sufficiently large $k$, and for each 3-ladder gadget that had one of its edges identified with $y_iz_i$, we instead identify the edge of the 3-ladder gadget with a different edge $v_jw_j$ of this $k$-ladder. All vertices in this construction have degree at most 3, except vertices $v_2^j$ in clause gadgets, which have degree 4. We omit the details.

\subsection{Planar Graphs}

The main result of this subsection is the following: 
\begin{theorem} \label{thm:npplanar}
    Deciding if a given planar graph admits a KT orientation is NP-complete. 
\end{theorem}

We prove this by introducing a new gadget, the crossing gadget shown in Figures \ref{fig:np-planar1} and \ref{fig:np-planar2}.
    \begin{figure}[p]
        \centering
        \includegraphics[]{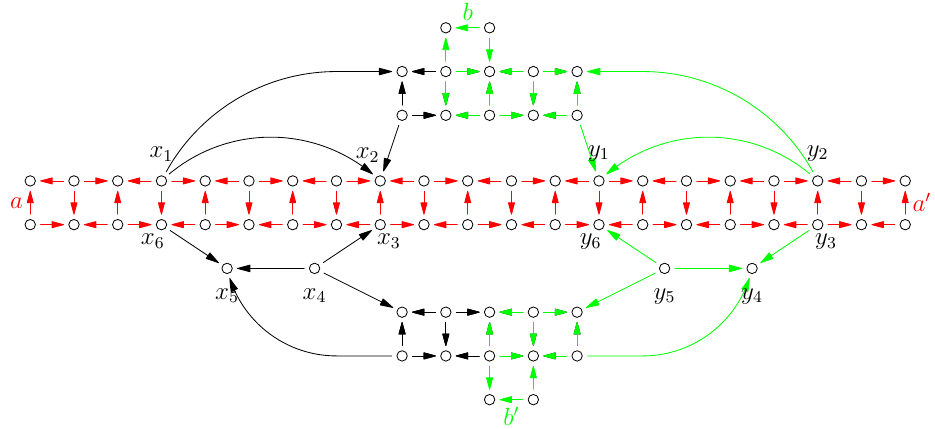}
        \caption{Crossing gadget.}
        \label{fig:np-planar1}
    \end{figure}
    \begin{figure}[p]
        \centering
        \includegraphics[]{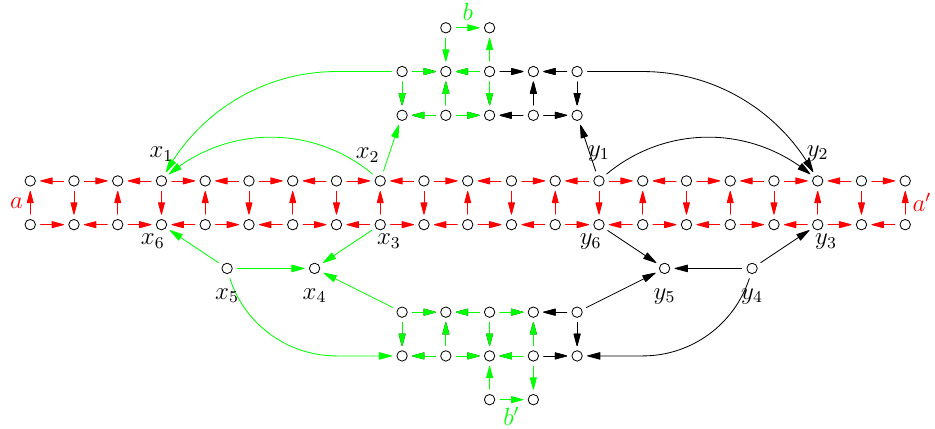}
        \caption{Another orientation of the crossing gadget.}
        \label{fig:np-planar2}
    \end{figure}
\begin{lemma}
Every KT orientation of the graph in Figures \ref{fig:np-planar1} and \ref{fig:np-planar2} satisfies that either both $a$ and $a'$ are oriented from bottom to top or vice versa, and likewise either both $b$ and $b'$ are oriented from right to left or vice versa. Furthermore, every orientation of the edges $a, a', b, b'$ such that both $a$ and $a'$ are oriented in the same direction, and both $b$ and $b'$ are oriented in the same direction, extends to a KT orientation of the gadget. 
\end{lemma}
\begin{proof}
    Note that $a$ and $a'$ will always have the same orientation because there is a ladder gadget between them consisting of an even number of four-cycles. Now we will prove that $b$ and $b'$ also have the same orientation. 
    \begin{itemize}
        \item 
            If $a$ is oriented from bottom to top and $b$ is oriented from right to left (or both are reversed) as in Figure \ref{fig:np-planar1}, then there is a directed path $y_3y_2y_1y_6$ (or directed in the opposite direction). So there is only one possible way to orient the rest of the six-cycle $y_1y_2y_3y_4y_5y_6$ as a KT orientation, namely orienting the remaining edges as $y_3 \rightarrow y_4$, $y_5 \rightarrow y_4$, and $y_5 \rightarrow y_6$. Since there is a series of four-cycles sharing edges from $y_4y_5$, this determines the orientation of $b'$ and implies that $b'$ is oriented from right to left. 
        \item 
            If $a$ is oriented from bottom to top and $b$ is oriented from left to right (or both are reversed) as in Figure \ref{fig:np-planar2}, we use the same observation as above for the six-cycle $x_1x_2x_3x_4x_5x_6$. 
    \end{itemize}
The second claim of the lemma follows from Figures \ref{fig:np-planar1} and \ref{fig:np-planar2} (possibly by reversing the orientations). 
\end{proof}

\begin{proof}[Proof of Theorem \ref{thm:npplanar}.]
    We describe how to modify the proof of Theorem \ref{thm:nphard1} using the crossing gadget. First, as in Theorem \ref{thm:npharddeg4}, we replace each edge $y_iz_i$ in the construction of Theorem \ref{thm:nphard1} with a $k$-ladder gadget, where $k$ is at least the number of occurrences of the variable $x_i$ in the monotone not-all-equal 3-SAT instance. For each 3-ladder gadget such that one of its edges was previous identified with $y_iz_i$, we instead identify it with an edge $v_jv_{j+1}$ of the $k$-ladder. Note that the resulting graph is planar; see Figure \ref{fig:ladder}.
    \begin{figure}
        \centering
        \includegraphics[width=0.6\textwidth]{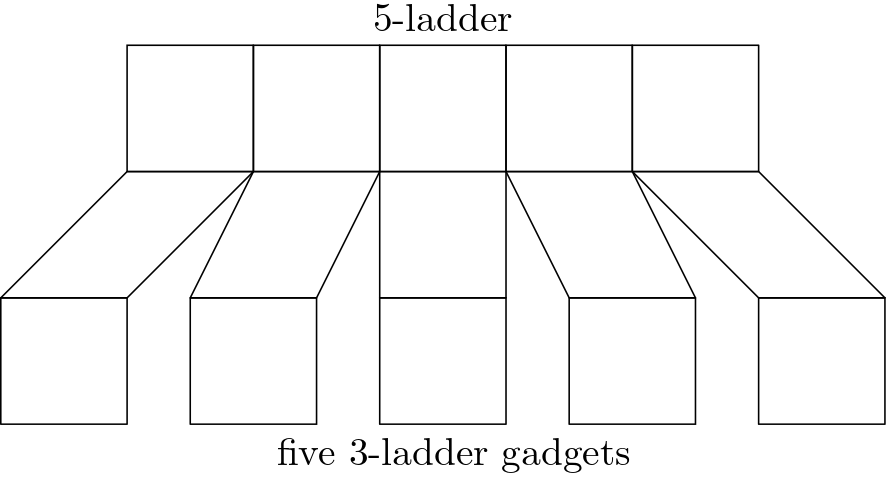}
        \caption{Attaching $3$-ladder gadgets to variable gadgets in Theorem \ref{thm:npplanar}.}
        \label{fig:ladder}
    \end{figure}
    
    The purpose of each 3-ladder gadget tape was to ``copy'' the orientation of an edge to another edge, but a longer ladder gadget would also serve this purpose. By replacing each 3-ladder gadget with a $k$-ladder gadget for some sufficiently large $k$, we can create a graph that can be embedded in the plane such that only ladder gadgets cross, and different crossings between ladder gadgets are at least some constant distance (in the graph) apart. Iteratively, using the gadgets in Figures \ref{fig:np-planar1} and \ref{fig:np-planar2}, we remove crossings between ladder gadgets. 

    Let us denote the resulting graph as $G$. As in Theorem \ref{thm:nphard1}, using Lemma \ref{lem:chain} and our observations about the crossing gadget mentioned above, it follows that if the given instance of monotone not-all-equal 3-SAT is not satisfiable, then $G$ does not admit a KT orientation. 

    To prove the converse, let us assume that the given instance of monotone not-all-equal 3-SAT is satisfiable. We orient each variable gadget, as well as ladder gadgets and crossing gadgets, according to the values of the variables (Figures \ref{fig:np-planar1} and \ref{fig:np-planar2} demonstrate the orientations of the crossing gadget). This gives an orientation to all edges except two edges per clause; these we orient by applying Lemma \ref{lem:clause} to each clause gadget. 

    Again, we notice that after deleting every vertex which is a source or sink such that all its neighbours are also sources or sinks by Lemma \ref{lem:sourcesink}, every component that remains contains at most 18 vertices (either 11 vertices from a clause gadget and its neighbours, or 18 vertices from one of the six-cycles $x_1, \dots, x_6$ or $y_1, \dots, y_6$ in a crossing gadget and its neighbours). As before, components containing a clause gadget are KT orientations. 

    It remains to consider components contained in a crossing gadget; but by inspection (or following the same line of reasoning as in Theorem \ref{thm:nphard1}), Figures \ref{fig:np-planar1} and \ref{fig:np-planar2} are KT orientations. This concludes the proof. 
    \end{proof}

\subsection{Graphs with Large Girth}

    In this subsection, we will show: 

    \begin{theorem} \label{thm:npgirth}
        Let $k \in \mathbb{N}$. Deciding whether a given graph has a KT orientation is NP-complete even when restricted to graphs of girth at least $k$. 
    \end{theorem}

    By Theorem \ref{thm:girth} we know that there is a graph with girth $k$ that admits no KT orientation. Let $G$ be such a graph with the minimal number of edges. Now, let $e = uv \in E(G)$; and let $G^- = G \setminus e$. We construct a graph $G^+$ from $G$ by subdividing the edge $e$ by introducing one or two new vertices, as follows: 
    \begin{itemize}
        \item[Case 1:] If there is a KT orientation $D$ of $G^-$ such that $D$ contains no directed path between $u$ and $v$, then we let $G^+$ be the graph obtained from $G^-$ by adding a vertex $x$, setting $y = v$, and adding edges $a = ux$ and $b = vy$ (see Figure \ref{fig:gplus}, left).   
        \item[Case 2:] Otherwise, we let $G^+$ be the graph arising from $G^-$ by adding vertices $x$ and $y$ as well as edges $a = ux$, $b=xy$ and $c = yv$ (see Figure \ref{fig:gplus}, right).   
    \end{itemize}
    Since $G^+$ is obtained from $G$ by subdividing an edge, it follows that $G^+$ has girth at least $k$. 

      \begin{figure}[H]
        \centering
        \includegraphics[]{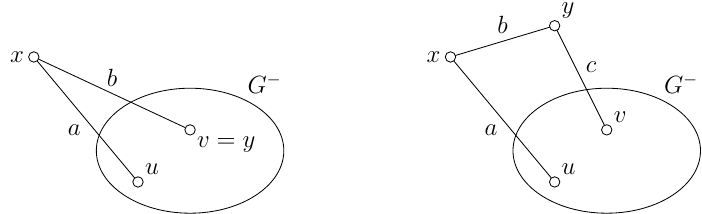}
        \caption{Construction of $G^+$.}
        \label{fig:gplus}
    \end{figure}
    
    \begin{lemma}\label{st:gplus}
        The graph $G^+$ admits a KT orientation such that $x$ is a source and a KT orientation such that $x$ is a sink. Furthermore, $x$ has degree 2 in $G^+$ and $x$ is either a source or a sink in every KT orientation of $G^+$.
    \end{lemma}

    \begin{proof}
    We first claim that there is a KT orientation $D$ of $G^+ \setminus x$ such that there is no directed path between $u$ and $y$. This is immediate from the construction of $G^+$ in the first case, by definition, so we may assume that $v \neq y$. By the minimality of $G$, we know that $G^-$ has a KT orientation $D'$. Since $D'$ is acyclic, and as we can reverse $D'$, we may assume that $D'$ contains no directed path from $u$ to $v$. Now, orienting $yv$ as $v \rightarrow y$, we obtain a KT orientation of $G^+ \setminus x$ with no directed path between $u$ and $y$. 

    We fix a KT orientation $D$ with no directed path between $u$ and $y$. We extend $D$ to the orientations $D_1$ and $D_2$ of $G^+$ by making $x$ a source and a sink, respectively. Our goal is to show that either $D_1$ or $D_2$ is a KT orientation. Then, by reversing all its edges, we obtain another KT orientation of $G^+$ and thus proving the first statement of Lemma \ref{st:gplus}.
    
    Let $i \in \{1, 2\}$. Then $x$ is either a source or a sink in $D_i$, and so not in the interior of a directed path. Therefore, if $D_i$ is not a KT orientation, then by Lemma \ref{lem:disjoint}, there is a vertex $z_i$ such that $D_i$ contains two distinct and internally disjoint directed paths between $x$ and $z_i$; let us call these paths $Q_i^1$ and $Q_i^2$ such that $Q_i^1$ contains $y$ and $Q_i^2$ contains $u$. Since $D$ contains no directed path between $u$ and $y$, it follows that $Q^1_i$ does not contain $u$, and likewise, $Q_i^2$ does not contain $y$. 

    Thus, $D$ contains two directed paths between $z_1$ and $z_2$, namely the concatenation $P^1$ of $Q_2^1$ (from $z_2$ to $y$) and $Q_1^1$ (from $y$ to $z_1$) and the concatenation $P^2$ of $Q_2^2$ (from $z_2$ to $u$) and $Q_1^2$ (from $u$ to $z_1$). To see that these paths are distinct, note that $P^1$ contains $y$, but not $u$, and $P^2$ contains $u$, but not $y$. This is a contradiction, as $D$ is a KT orientation. Thus we have proved the first statement of Lemma \ref{st:gplus}. 
    
    For the second statement, clearly $x$ has degree 2 in $G^+$. Now, suppose for a contradiction that $D$ is a KT orientation of $G^+$ such that $u \rightarrow x$ and $x \rightarrow y$ in $D$ (the reverse case is analogous). We again consider whether $v=y$ or not. If $v=y$, then we obtain a KT orientation of $G$ by orienting edges as in $D$, and letting $u \rightarrow y$; but this contradicts the assumption that $G$ does not admit a KT orientation. 
    
    If $v \neq y$, then we note that $D \setminus x$ has no directed path between $u$ and $v$: On the contrary suppose that $P$ is such a path. Then, if $P$ is a path from $u$ to $v$, it follows that: 
    \begin{itemize}
        \item If $v \rightarrow y$ in $D$, then $D$ contains two distinct directed paths from $u$ to $y$, namely $u, x, y$ and $u, P, v, y$, a contradiction. 
        \item If $y \rightarrow v$ in $D$, then $D$ contains two distinct directed paths from $u$ to $v$, namely $u, x, y, v$ and $u, P, v$, a contradiction. 
    \end{itemize}
    Therefore, $P$ is a path from $v$ to $u$. Now, 
    \begin{itemize}
        \item If $y \rightarrow v$ in $D$, then $u, x, y, v, P, u$ is a directed cycle in $D$, a contradiction.  
        \item If $v \rightarrow y$ in $D$, then $D$ contains two distinct directed paths from $v$ to $y$, namely $v, y$ and $v, P, u, x, y$, and thus a contradiction. 
    \end{itemize}
    This proves the second statement of Lemma \ref{st:gplus}. 
    \end{proof}

    \begin{figure}[t]
        \centering
        \includegraphics[]{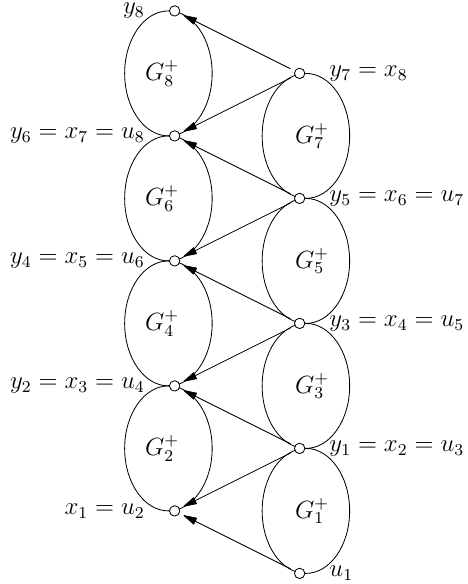}
            \caption{The graph $G^8$.}
            \label{fig:np_copy}
        \end{figure}

    \medskip
    
    Lemma \ref{st:gplus} allows us to create an analogue of the $k$-ladder, as follows (see Figure \ref{fig:np_copy}). We take copies $G_1^+, \dots, G_l^+$ of $G^+$, denoting the copies of the vertices $u, x, y$ in $G_i^+$ as $u_i, x_i, y_i$. Then, we identify the following sets of vertices: 
    \begin{itemize}
        \item $\{u_i, x_{i-1}, y_{i-2}\}$ for $i \in \{3, \dots, l\}$; 
        \item $\{x_1, u_2\}$; 
        \item $\{x_l, y_{l-1}\}$. 
    \end{itemize}
    We denote the resulting graph by $G^l$. Note that $G^l$ has girth at least $k$ (since gluing together graphs along clique cutsets does not decrease girth). By Observation \ref{obs:cc}, it follows that $G^l$ admits a KT orientation, and if $l$ is even, then (using Lemma \ref{st:gplus}) in every KT orientation, one of the following is true:
    \begin{itemize}
        \item $y_{l-1} \rightarrow y_l$ and $u_1 \rightarrow u_2$; or
        \item $y_l \rightarrow y_{l-1}$ and $u_2 \rightarrow u_1$. 
    \end{itemize}
    Next, we create an analogue of the clause gadgets in Theorem \ref{thm:nphard1}. We replace each five-cycle by a longer cycle, by subdividing the edge between $v_1$ and $v_2$, creating a path $P$. Then, for every two consecutive edges of $P$, we add a copy of $G^l$ that ensures that these two edges are oriented in the direction. It follows that $P$ becomes a directed path in every KT orientation of the graph we create. By choosing $l$ sufficiently large, we ensure that the resulting graph has large girth. 

    \begin{proof}[Proof of Theorem \ref{thm:npgirth}]
    Let $I$ be an instance of monotone not-all-equal 3-SAT. Let $H$ be the graph obtained by applying the construction from Theorem \ref{thm:nphard1} to $I$, and then modifying it using $G^l$ (for some sufficiently large $l$) instead of each 3-ladder gadget, as well as the modified clause gadgets from the previous paragraph. Note that for every fixed $k$, we can arrange that $H$ has girth at least $k$. From our construction and the correctness of the construction in Theorem \ref{thm:nphard1}, it immediately follows that if $H$ admits a KT orientation, then $I$ has a satisfying assignment. Conversely, suppose that $I$ has a satisfying assignment. Let $K$ be the orientation resulting from the valid truth assignment for $I$. We orient the input edges according to the assignment and the $G^l$ gadgets between them and the clause gadgets in such a way that all the copies of $G^{+}$ are oriented according to a KT orientation, which is possible by the construction of $H$, the fact that $I$ is satisfying, and the property from Lemma \ref{st:gplus} of $G^+$.  

We now observe that in each KT orientation of a copy of $G^+$, there is no directed path between its vertex $u$ and its vertex $y$; otherwise, along with its vertex $x$, we would have two distinct directed paths between two vertices of $G^+$, a contradiction. 
    
    Therefore, a directed path which contains a vertex of a $G^l$, and not in the first or last three copies of $G^+$ within the $G^l$, has both its ends within the same copy of $G^+$; but each copy of $G^+$ is a KT orientation, a contradiction. 

    Supposing now that $K$ contains two vertices $x$ and $y$ as well as two distinct directed paths $P$ and $Q$ between them. It follows from that neither $P$ nor $Q$ contains a vertex which is in a copy of $G^l$ but not in the first or last three copies of $G^+$ within that copy of the gadget $G^l$. Let $K'$ be the induced subgraph of $G$ in which we remove all but the first and last three copies of $G^+$ within each $G^l$. Then $x, y, P$ and $Q$ are contained in $K'$. We apply Observation \ref{obs:cc} to each component of $K'$, and find that either some copy of $G^+$ is not oriented according to a KT orientation (contrary to the definition of $K$), or one of the clause gadget cycles is not oriented according to a KT orientation (again contrary to the definition of $K$). This contradiction shows that $K'$ is a KT orientation, contrary to the choice of $P$ and $Q$. Therefore, $K$ is a KT orientation. 

    \end{proof}

%% file: three_v2/three_v2.tex
\section{Graphs with Small Maximum Degree}
\label{sec:algorithm}
    
    Our main result in this section is the following: 
    \begin{theorem} \label{thm:6main}
    There is a polynomial-time algorithm with the following specifications: 
    \begin{itemize}
        \item Input: A graph $G$ of maximum degree at most 3. 
        \item Output: A KT orientation of $G$, or a determination that none exists. 
    \end{itemize}
    \end{theorem}

    To prove Theorem \ref{thm:6main}, we need the following: 
    \begin{theorem}[Brooks \cite{brooks}] \label{thm:brooks}
        Let $G$ be a connected graph of maximum degree at most 3 such that $G \neq K_4$. Then $G$ is 3-colourable.         
    \end{theorem}
    The proof of Theorem \ref{thm:brooks} given by Lov\'asz in \cite{lovasz1975three} yields a polynomial-time algorithm for finding such a colouring. Our main tool in proving Theorem \ref{thm:6main} is a careful analysis of the four-cycles in the input graph, with the goal of applying the following lemma: 
    \begin{lemma} \label{lem:3col4cycle}
        Let $G$ be a triangle-free graph, and let $f$ be a 3-colouring of $G$ such that for every four-cycle $C$ of $G$, we have $|f(V(C))| = 2$ (in other words, every four-cycle is 2-coloured). Then $G$ admits a KT orientation, obtained by orienting each edge $uv$ with $f(u) < f(v)$ as $u \rightarrow v$. 
    \end{lemma}
    \begin{proof}
        We orient the edges of $G$ as follows: for an edge $uv$, if $f(u) < f(v)$, we orient $uv$ as $u \rightarrow v$, and  otherwise, we orient it as $v \rightarrow u$. This orientation is acyclic, and moreover, every directed two-edge path contains a vertex of each of the three colours, and there is no directed path with more than three vertices. Suppose now that the resulting orientation, $D$, is not a KT orientation. Let $u, v \in V(D)$ with two distinct paths $P$ and $Q$ from $u$ to $v$. Then, since $G$ is triangle-free and $D$ has no directed path with more than three vertices, it follows that each of $P$ and $Q$ has exactly three vertices. But then the four-cycle with vertex set $V(P) \cup V(Q)$ in $G$ receives all three colours, contrary to our assumption about the colouring $f$. This concludes the proof. 
    \end{proof}

    Given a graph $G$, an edge $e \in E(G)$ is a \emph{bridge} of $G$ if $G \setminus e$ has more components than $G$ does. The following is straightforward: 

    \begin{lemma} \label{lem:bridge}
        Let $G$ be a graph and let $e$ be a bridge of $G$. Then $G$ admits a KT orientation if and only if every component of $G \setminus e$ does. 
    \end{lemma}

   We first analyze components of four-cycles: 

    \begin{lemma} \label{lem:hcomp}
     Let $G$ be a connected triangle-free graph of maximum degree 3. We define a hypergraph $H$ as follows: 
        \begin{itemize}
            \item $V(H) = V(G)$; 
            \item For every four-cycle in $G$, we add an edge containing its four vertices, to $H$. 
        \end{itemize}
     Let $C$ be a component of the hypergraph $H$ containing at least one edge. Then the induced subgraph $G[V(C)]$ is one of the following graphs:
            \begin{itemize}
                \item A cube minus an edge, the graph on the left in Figure \ref{fig:partcube2}.
                \item A cube minus a vertex, the graph on the left and in the middle in Figure \ref{fig:partcube}. 
                \item A $k$-ladder, such as the graph on the left in Figure \ref{fig:chaindeg3}.
                \item A $k$-ladder with exactly one of $v_1w_k$ and $v_1v_k$ as an additional edge (where $v_1, \dots, v_k, w_1, \dots, w_k$ are as in Definition \ref{def:kladder}). 
                \item $K_{2,3}$.
                \item $K_{3,3}-e$, the graph obtained from $K_{3,3}$ by removing an edge. 
                \item A $k$-ladder with two additional edges, namely either $\{v_1v_k, w_1w_k\}$ or $\{v_1w_k, w_1v_k\}$ (where $v_1, \dots, v_k, w_1, \dots, w_k$ are as in Definition \ref{def:kladder}).
                \item A cube, the graph in Figure \ref{fig:cube}. 
        \begin{figure}[H]
                    \centering
                        \includegraphics[]{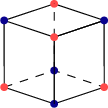}
                    \caption{A cube.}
                    \label{fig:cube}
                \end{figure}
                \item $K_{3,3}$. 
            \end{itemize}
        \end{lemma}
        \begin{proof}
Since $G$ has maximum degree at most 3, we observe: 
        \begin{observation}\label{obs:share}
            Two distinct edges of $H$ share either $0, 2$ or $3$ vertices. The corresponding four-cycles in $G$ either have no vertices in common, or an edge in common, or a 2-edge path in common. 
        \end{observation}
        
        Since, by assumption, $H$ contains at least one edge, it follows that $G[V(C)]$ contains at least one four-cycle. Moreover, since $G$ is connected and $V(C) \neq V(G)$, it follows that not all vertices of $G[V(C)]$ have degree 3. 
        
        If $C$ has a single edge, then $G[V(C)]$ is a 2-ladder. If $C$ has at least two edges, then there are two cases by Observation \ref{obs:share}:
        \begin{itemize}
            \item Two edges of $C$ share exactly two vertices; or
            \item Two edges of $C$ share exactly three vertices. 
        \end{itemize}
        Let us first consider the second case. Then $G[V(C)]$ contains $K_{2,3}$ as a subgraph, and since $G$ is triangle-free, this is an induced subgraph. If $C$ contains exactly these five vertices, then the fifth outcome of the lemma holds, and we are done. 

        Let $X$ be the vertex set of an induced $K_{2,3}$ in $G[V(C)]$, and let $x_1, x_2$ be the vertices of degree 3 in $G[X]$, and let $y_1, y_2, y_3$ be the vertices of degree 2 in $G[X]$. Then, since $C$ is connected and by Observation \ref{obs:share}, $C$ contains an edge $e$ such that $|e \cap X| \in \{2, 3\}$. Since $N_G(x_1) = N_G(x_2) = \{y_1, y_2, y_3\}$ and the maximum degree of $G$ is 3, it follows that $G[e]$ contains two edges in $G[X]$, and so $|e \cap X| = 3$. We conclude that $e$ contains a vertex $x_3$ adjacent to at least two of $y_1, y_2, y_3$. If $x_3$ is adjacent to all of $y_1, y_2, y_3$, then $G = K_{3,3}$ and $G[V(C)] = K_{3,3}$, and the last outcome holds. Therefore, and by symmetry, we may assume that $y_1, y_2$ are adjacent to $x_3$, and $y_3$ is non-adjacent to $x_3$. It follows that  $G[V(C)]$ contains $K_{3,3} - e$ as an induced subgraph (see Figure \ref{fig:k33minuse}). If $V(C) = X \cup \{x_3\}$, then $G[V(C)]$ is isomorphic to $K_{3,3}-e$, as desired. Otherwise, since $y_3$ and $x_3$ are the only vertices in $X \cup \{x_3\}$ who may have a neighbour outside of $X \cup \{x_3\}$ in $G$, it follows that there is a four-cycle in $G$ containing $y_3$ and $x_3$. Since $y_3, x_3$ are non-adjacent, it follows that they have at least two common neighbours. Since the neighbours of $y_3, x_3$ in $X$ are not common neighbours of $y_3$ and $x_3$, it follows that $y_3, x_3$ have degree at least four. This is a contradiction, and concludes the analysis of the second case. 

        Now we may assume that no two edges of $C$ share three vertices. If $C$ has exactly two edges, then $G[V(C)]$ is a 3-ladder, and the lemma holds. 
        
        If follows that $C$ has at least three edges. Since $C$ is connected, there is an edge of $C$ that shares 2 vertices with 2 different edges. Since $G$ has maximum degree at most 3, there are two possible configurations for a four-cycle sharing two vertices with two other four-cycles: a cube minus a vertex, as depicted in Figure \ref{fig:partcube} (left) and the 4-ladder which is shown in Figure \ref{fig:chaindeg3} (left). 

        Let us first consider the case that $G[V(C)]$ contains a cube minus a vertex (and note that since $G$ has maximum degree 3 and is triangle-free, it follows that this is an induced subgraph). Let $X$ be the vertex set of a cube minus a vertex in $G[V(C)]$. If $V(C) = X$, then we are done. Otherwise, let $x_1, x_2, x_3$ be the three vertices of degree 2 in $G[X]$. If there is a vertex in $G$ adjacent to all of $x_1, x_2, x_3$, then $G$ (and therefore $G[V(C)]$) is a cube, and the second-to-last outcome holds, as desired. Now we may assume that no vertex in $G$ is adjacent to all of $x_1, x_2, x_3$. Since $C$ contains an edge $e$ such that $|e \cap X| \in \{2, 3\}$, it follows that two of $x_1, x_2, x_3$, say $x_1$ and $x_2$, have a common neighbour, say $y$, in $G \setminus X$ (and so $e$ is the four-cycle consisting of $x_1, y, x_2$ and the common neighbour of $x_1$ and $x_2$ in $G[X]$). 

        Now $G[X \cup \{y\}]$ is isomorphic to a cube minus an edge (see Figure \ref{fig:partcube2}). Moreover, the only vertices of $X \cup \{y\}$ which may have a neighbour (in $G$) outside of $X \cup \{y\}$ are $x_3$ and $y$. Since paths from $x_3$ to $y$ with interior in $X$ have length at least three, and $x_3$ is non-adjacent to $y$, it follows that no four-cycle in $G$ contains both $x_3$ and $y$, and so every four-cycle in $G$ which contains a vertex in $X \cup \{y\}$ is contained in $X \cup \{y\}$. Therefore, $V(C) = X \cup \{y\}$, and the statement of the lemma holds.

        Now we may assume that $G[V(C)]$ contains a 4-ladder, and does not contain a cube minus a vertex. Let $k$ be maximum such that $G[V(C)]$ contains a $k$-ladder, and let $v_1, \dots, v_k, w_1, \dots, w_k$ denote the vertex set of a $k$-ladder in $G[V(C)]$, with the notation as in Definition \ref{def:kladder}. 

        Suppose first that $V(C) \neq X$. Then, some four-cycle $Q$ in $G$ shares an edge with $G[X]$, but such that $V(Q) \not\subseteq X$. Since $G$ has maximum degree 3, it follows that $Q$ contains $v_1w_1$ or $v_kw_k$, and by symmetry, we may assume the former. If $Q$ contains no other vertex of $X$, then $X \cup V(Q)$ contains a $(k+1)$-ladder, contrary to the choice of $X$. Therefore, and since $V(Q) \not\subseteq X$, it follows that $Q$ contains exactly one of $v_k, w_k$; by symmetry, we may assume the former. But now $v_k$ has degree at least four: two neighbours in $Q$, along with $w_k$ and $v_{k-1}$. This is a contradiction, proving $X = V(C)$. 

        If $G[X]$ is an induced $k$-ladder, then the statement holds. Since $G$ has maximum degree 3, the possible only additional edges in $G[X]$ are $v_1w_k$, $v_1v_k$, $w_1v_k$, $w_1w_k$. If exactly one of these edges is present, then (possibly switching the roles of $v_i$ and $w_i$ for all $i$) the fourth outcome holds. Therefore, exactly two of these edges are present (using that $G$ has maximum degree 3). But now every vertex in $G[X]$ has degree 3, and so $V(C) = V(G)$, and the third-to-last outcome holds. This concludes the proof. 
        \end{proof}

    Let us briefly describe our proof strategy. We start with the following:
    
        \begin{observation} \label{obs:onecomp}
            Let $G, H$ be as in Lemma \ref{lem:hcomp}. Let $C$ be a component of $H$. If one of the last three outcomes holds for $C$, then $H$ has exactly one component $C$ $\mathrm{(}$and so $V(C) = V(G)$$\mathrm{)}$. 
        \end{observation}
        \begin{proof}
            This follows immediately by observing that every vertex in $G[V(C)]$ has degree 3. 
        \end{proof}

        Next, let's prove that we can disregard cases in which $H$ has only a small number of components. For that, we need the following: 
    
        \begin{observation}\label{obs:comp}
            Let $G, H$ be as in Lemma \ref{lem:hcomp}. Let $C$ be a component of $H$. Then the subgraph of $G$ induced by the vertex set of $C$, $G[V(C)]$, has at most eight KT orientations, and we can compute all KT orientations of $G[V(C)]$ in polynomial time (or decide that none exists). 
        \end{observation}
        \begin{proof}
            Note that a four-cycle has exactly two KT orientations. Each of the outcomes of Lemma \ref{lem:hcomp} therefore has at most eight KT orientations (orienting one edge in a four-cycle forces the orientations of other edges in four-cycles, and then we have at most two additional edges for which to choose an orientation). By Lemma \ref{lem:ktalg}, we can decide in polynomial time which, if any, of these up to eight orientations are KT orientations, and output them. 
        \end{proof} 
        We point out that ``eight'' can be improved to ``four,'' since in the cases with two additional edges, all edges are in four-cycles. 
        
        Using Observation \ref{obs:comp}, if $H$ has few components, we can check for a KT orientation via brute-force. Otherwise, we will replace each component of $H$ by one or two vertices, 3-colour the resulting graph, and then undo the replacement to obtain a 3-colouring of $G$ such that every four-cycle is two-coloured, as Lemma \ref{lem:3col4cycle} requires. There is one exception to this: It does not work for additional edges (those is the fourth and seventh bullet point of Lemma \ref{lem:hcomp}). We will show that orienting them arbitrarily after applying the above strategy to the rest of the graph is sufficient. In particular, our result implies that if $G$ is connected and $H$ has at least five components, then $G$ admits a KT orientation. 

       The following cases formalizes which ``additional edges'' require special treatment.     
        Let $G$ be a connected triangle-free graph, and let $H$ be as in Lemma \ref{lem:hcomp}. We say that an edge $e \in E(G)$ is \emph{exceptional} if there is a component $C$ of $H$ such that $G[V(C)]$ is a $k$-ladder with exactly one additional edge, which is not in a four-cycle, and, denoting the vertices of the $k$-ladder as $v_1, \dots, v_k, w_1, \dots, w_k$ as in Definition \ref{def:kladder}, one of the following holds: 
        \begin{itemize}
            \item $k$ is odd and $e = v_1v_k$; or
            \item $k$ is even and $e = v_1w_k$.
        \end{itemize} See Figure \ref{fig:exceptional} for an example. 

        \begin{observation} \label{obs:bipcomp}
            Let $G, H$ be as in Lemma \ref{lem:hcomp}. Suppose that $H$ has at least two components. Let $G^*$ be obtained from $G$ by removing all exceptional edges. Then, for each component $C$ of $H$, the graph $G^*[V(C)]$ is bipartite and connected. 
        \end{observation}
        \begin{proof}
            This is immediate from Lemma \ref{lem:hcomp} and Observation \ref{obs:onecomp}. 
        \end{proof}

        Finally, from Lemma \ref{lem:hcomp}, the following is immediate: 
        \begin{observation}\label{obs:exce}
        Let $G, H$ be as in Lemma \ref{lem:hcomp}. Let $v \in V(G)$. Then at least one of the following holds: 
        \begin{itemize}
            \item $v$ is not incident to an exceptional edge; or
            \item all neighbours of $v$ in $G$ are in the same component in $H$ as $v$ is. 
        \end{itemize}
        \end{observation}

        We are now ready to prove the main result of this section. 
    
    \begin{proof}[Proof of Theorem \ref{thm:6main}.]
        We may assume that $G$ is triangle-free; if not, we output that $G$ does not admit KT orientation. By considering each component of $G$ separately, we may assume that $G$ is connected. By Lemma \ref{lem:bridge}, we may further assume that $G$ does not contain any bridge, and so $G$ is $2$-edge-connected. 

        Let $H$ be as in Lemma \ref{lem:hcomp}. If $H$ has at most four components, then $G$ contains at most 8 edges that do not have both ends within the same component of $H$ (using that for each component $C$ of $H$, at most four edges of $G$ have exactly one end in $V(C)$ by Lemma \ref{lem:hcomp}). Using Observation \ref{obs:comp}, we compute a list of all possible KT orientations for $G[V(C)]$ for each component $C$ of $H$, and for each choice of a KT orientation of each component of $H$, we try all orientations of the remaining at most 8 edges. This leads to at most $8^4 \cdot 2^8$ orientations in total; then, by Lemma \ref{lem:ktalg}, we check if any of them are KT orientations, and return the result. 
        
        Now we may assume that $H$ has at least five components. Let $G^*$ be the graph obtained from $G$ by removing exceptional edges. We now construct a graph $H$, iteratively, as follows. For every component $C$ of $H$ with $|V(C)| > 1$, the graph $G^*[V(C)]$ has a unique bipartition (by Observation \ref{obs:bipcomp}), say $(A_C, B_C)$. We replace $V(C)$ by two vertices $a_C$ and $b_C$ such that: 
        \begin{itemize}
            \item $a_Cb_C$ is an edge; 
            \item $a_C$ is adjacent to all neighbours of $A_C$ outside of $C$; 
            \item $b_C$ is adjacent to all neighbours of $B_C$ outside of $C$; 
            \item if the only neighbour of $a_C$ is $b_C$, delete $a_C$; and 
            \item if the only neighbour of $b_C$ is $a_C$, delete $b_C$. 
        \end{itemize}
        Denote the resulting graph by $G'$. 

        Next, we will examine all outcomes of Lemma \ref{lem:hcomp}, and check that all vertices of $G'$ have degree at most 3. Let $C$ be a component of $H$. The following, using Observation \ref{obs:onecomp}, are the possibilities for $G^*[V(C)]$: 
        \begin{itemize}
            \item A cube minus an edge, the graph on the left in Figure \ref{fig:partcube2}. Both $a_C$ and $b_C$ have degree at most two.  

                \begin{figure}[H]
                    \centering
                        \includegraphics[]{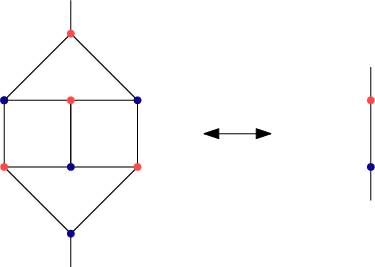}
                    \caption{The cube minus an edge, plus possible edges to the remainder of the graph (left), and the replacement $a_C$ and $b_C$ (right).}
                    \label{fig:partcube2}
                \end{figure}
            
            \item A cube minus a vertex, the graph shown both on the left and in the middle in Figure \ref{fig:partcube} (the graph is shown twice to highlight similarities and differences with other cases). Note that one of $a_C, b_C$ has no neighbours outside $C$ and hence will be deleted; the remaining vertex has degree at most three, since there are at most three edges with exactly one end in $V(C)$. 
            
                \begin{figure}[H]
                    \centering
                        \includegraphics[]{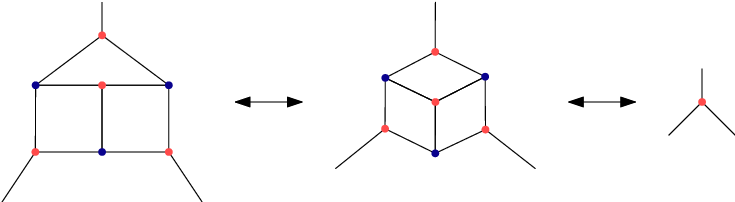}
                    \caption{The cube minus a vertex, plus possible edges to the remainder of the graph (left and middle), and the replacement (right); note that one of $a_C, b_C$ was deleted.}
                    \label{fig:partcube}
                \end{figure}

            \item A $k$-ladder, such as the graph on the left in Figure \ref{fig:chaindeg3}. Both $a_C$ and $b_C$ have degree at most three. 
            
            \begin{figure}[H]
                \centering
                    \includegraphics[]{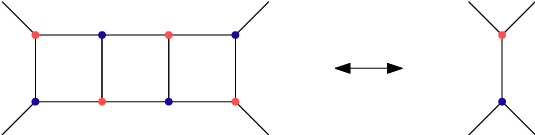}
                \caption{A $4$-ladder, plus possible edges to the remainder of the graph (left), and the replacement $a_C$ and $b_C$ (right).} \label{fig:chaindeg3}
            \end{figure}
            
            \item $K_{3,3}-e$, the graph on the left in Figure \ref{fig:k33minuse}. Both $a_C$ and $b_C$ have degree at most two. 

            \begin{figure}[H]
                \centering
                    \includegraphics[]{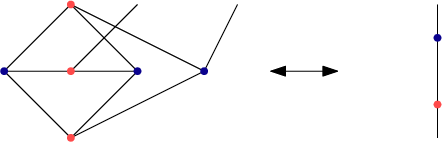}
                \caption{$K_{3,3}-e$, plus possible edges to the remainder of the graph (left), and the replacement $a_C$ and $b_C$ (right)} \label{fig:k33minuse}
            \end{figure}
   
            \item $K_{2,3}$, the graph on the left in Figure \ref{fig:k23}. Note that one of $a_C, b_C$ has no neighbours outside $C$ and hence will be deleted; the remaining vertex has degree at most three. 
            
            \begin{figure}[H]
                \centering
                    \includegraphics[]{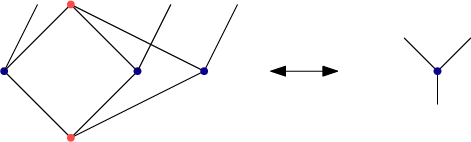}
                \caption{$K_{2,3}$, plus possible edges to the remainder of the graph (left), and the replacement (right); note that one of $a_C, b_C$ was deleted.} \label{fig:k23}
            \end{figure}
        \end{itemize} 

        Therefore, the graph $G'$ has maximum degree 3, and by construction, $G'$ is connected. Moreover, the number of vertices of $G'$ is at least the number of components of $H$, and so $|V(G')| > 4$. By Theorem \ref{thm:brooks}, we conclude that $G'$ has a 3-colouring, say $f$.

        Then, $G^*$ has a 3-colouring $f'$ arising  from $f$, by extending the colouring as shown in Figures \ref{fig:partcube2}--\ref{fig:k23}, that is, for every component $C$, we proceed as follows: 
        \begin{itemize}
            \item If $a_C$ was deleted, restore $a_C$ and assign to it a colour distinct from that of $b_C$;
            \item If $b_C$ was deleted, restore $b_C$ and assign to it a colour distinct from that of $a_C$;
            \item Assign the colour of $a_C$ to every vertex in $A_C$, and the colour of $b_C$ to every vertex in $B_C$.
        \end{itemize}
        By construction, every four-cycle in $G^*$ is coloured with two colours by $f'$. 

        Let $D$ be the KT orientation of $G^*$ arising from Lemma \ref{lem:3col4cycle}. It remains to find a suitable orientation for the exceptional edges; we orient them arbitrarily and claim that this yields a KT orientation $D'$ of $G$. 

    Suppose not; then there exist vertices $u$ and $v$ in $G$ with two internally disjoint paths $P$ and $Q$ between them. The union of $P$ and $Q$ is a cycle $S$ in $G$ such that the orientations of edges of $S$ change at most twice along $S$; in other words, at most two vertices of $S$ are sources or sinks with respect to the set of edges of $S$. Since $D$ is a KT orientation of $G$, it follows that $S$ contains an exceptional edge $e$. Let $C$ be the component of $H$ containing the ends of $e$; then $G^*[V(C)]$ is a $k$-ladder. If $V(S) \subseteq C$, then, since $e$ is not contained in a four-cycle of $G$, it follows that $S$ contains at least three vertices which are sources or sinks (in $C$ and hence in $S$), and so at least three direction changes, as desired. 
    
        It follows that $S$ is not contained in $C$, and so $S$ contains exactly two edges with one end in $C$ and the other in $V(G) \setminus C$, say $e'$ and $e''$; note that there are exactly two such edges. If $e'$ and $e''$ share an end, say $v$, it follows that $v \not\in V(C)$. Now, either $V(G) = \{v\} \cup V(C)$ (and so $H$ has at most two components, a case we already covered), or $v$ has degree 3 and is incident with an edge $j \neq e', e''$; but now $j$ is a bridge in $G$ (with one component of $G \setminus e$ being $G[\{v\} \cup V(C)]$), again a contradiction. Therefore, $e'$ and $e''$ are disjoint. 

        Let $e = yz$. By Observation \ref{obs:exce}, it follows that $e', e''$ and $e$ are pairwise disjoint. Let $x$ be the end of $e'$ in $C$, and let $w$ be the end of $e''$ in $V(C)$. See Figure \ref{fig:exceptional}. 
        \begin{figure}[H]
            \centering
            \includegraphics[width=0.7\linewidth]{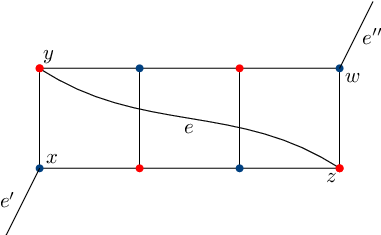}
            \caption{An exceptional edge $yz$ for a $4$-ladder. Note that 4 is the smallest $k$ such that a $k$-ladder can have an exceptional edge.}
            \label{fig:exceptional}
        \end{figure}
        
        Since $y, z$ are sources or sinks in $D$, it follows that one of $y$ and $z$ is a source or sink in $D'$. Therefore, a direction change of $S$ occurs at one of $y$ and $z$. 

        Moreover, by Observation \ref{obs:exce}, it follows that none of the vertices contained in $e'$ and $e''$ are incident with exceptional edges. Since at least one end of $e'$ is not coloured 2, and at least one end of $e''$ is not coloured 2, it follows that one end of $e'$ and one end of $e''$ is a source or a sink in $D'$. Consequently, $S$ contains at least three vertices which are sources or sinks (namely, one end of each of the three disjoint edges $e, e', e''$). This shows that $S$ has more than two direction changes, contrary to our assumption. So $D$ is a KT orientation. 
    \end{proof}

%% file: sample.bib
@article{colorful_ind_subg,
	title = {Colorful induced subgraphs},
	journal = {Discrete Mathematics},
	volume = {101},
	number = {1},
	pages = {165-169},
	year = {1992},
	issn = {0012-365X},
	url = {https://www.sciencedirect.com/science/article/pii/0012365X9290600K},
	author = {Hal A. Kierstead and William T. Trotter} 
}

@article{lovasz1975three,
  title={Three short proofs in graph theory},
  author={Lov{\'a}sz, L{\'a}szl{\'o}},
  journal={Journal of Combinatorial Theory, Series B},
  volume={19},
  number={3},
  pages={269--271},
  year={1975},
  publisher={Academic Press}
}

@article{larsen1995fractional,
  title={The fractional chromatic number of Mycielski's graphs},
  author={Larsen, Michael and Propp, James and Ullman, Daniel},
  journal={Journal of Graph Theory},
  volume={19},
  number={3},
  pages={411--416},
  year={1995},
  publisher={Wiley Online Library}
}

@article{duffus1995computational,
  title={On the computational complexity of ordered subgraph recognition},
  author={Duffus, Dwight and Ginn, Mark and R{\"o}dl, Vojt{\v{e}}ch},
  journal={Random Structures \& Algorithms},
  volume={7},
  number={3},
  pages={223--268},
  year={1995},
  publisher={Wiley Online Library}
}

@book{gyarfas1985problems,
  title={Problems from the world surrounding perfect graphs},
  author={Gy{\'a}rf{\'a}s, Andr{\'a}s},
  number={177},
  year={1985},
  publisher={MTA Sz{\'a}m{\'\i}t{\'a}stechnikai {\'e}s Automatiz{\'a}l{\'a}si Kutat{\'o} Int{\'e}zet}
}

@article{chi_boundedness,
  title={A survey of $\chi$-boundedness},
  author={Scott, Alex and Seymour, Paul},
  journal={Journal of Graph Theory},
  volume={95},
  number={3},
  pages={473--504},
  year={2020},
  publisher={Wiley Online Library},
    url = "https://onlinelibrary.wiley.com/doi/full/10.1002/jgt.22601"
}

@article{twincut,
	title={A tamed family of triangle-free graphs with unbounded chromatic number},
	author={Bonnet, {\'E}douard and Bourneuf, Romain and Duron, Julien and Geniet, Colin and Thomass{\'e}, St{\'e}phan and Trotignon, Nicolas},
	journal={arXiv preprint arXiv:2304.04296},
	year={2023},
	url = "https://arxiv.org/abs/2304.04296"
}

@article{trianglefree_contraexample,
	url = {https://doi.org/10.1016%2Fj.jctb.2022.09.001},
	year = {2023},
	month = {1},
	publisher = {Elsevier {BV}},
	volume = {158},
	pages = {63--69},
	author = {Alvaro Carbonero and Patrick Hompe and Benjamin Moore and Sophie Spirkl},
	title = {A counterexample to a conjecture about triangle-free induced subgraphs of graphs with large chromatic number},
	journal = {Journal of Combinatorial Theory, Series B}
}

@article{brianski2024separating,
  title={Separating polynomial $\chi$-boundedness from $\chi$-boundedness},
  author={Bria{\'n}ski, Marcin and Davies, James and Walczak, Bartosz},
  journal={Combinatorica},
  volume={44},
  number={1},
  pages={1--8},
  year={2024},
  publisher={Springer}
}

@article{large_chromatic,
    author={Girão, António and Illingworth, Freddie and Powierski, Emil and Savery, Michael and Scott, Alex and \; Tamitegama, Youri and Tan, Jane},
	title={Induced subgraphs of induced subgraphs of large chromatic number}, 
    journal={Combinatorica},
	year={2024},
    volume={44},
    pages={37--62},
    url = "https://doi.org/10.1007/s00493-023-00061-4"
}

@article{fractional_chromatic_zykov,
	title = {The fractional chromatic number of Zykov products of graphs},
	journal = {Applied Mathematics Letters},
	volume = {24},
	number = {4},
	pages = {432-437},
	year = {2011},
	issn = {0893-9659},
	url = {https://www.sciencedirect.com/science/article/pii/S0893965910003915},
	author = {Pierre Charbit and Jean Sébastien Sereni},
}

@article{brooks,
  title={On colouring the nodes of a network},
  author={Brooks, R. L.},
  journal={Mathematical Proceedings of the Cambridge Philosophical Society},
  volume={37},
  number={2},
  pages={194--197},
  year={1941},
  doi={10.1017/S030500410002168X}
}

@inproceedings{nae3sat,
author = {Schaefer, Thomas J.},
title = {The Complexity of Satisfiability Problems},
year = {1978},
isbn = {9781450374378},
publisher = {Association for Computing Machinery},
address = {New York, NY, USA},
url = {https://doi.org/10.1145/800133.804350},
doi = {10.1145/800133.804350},
booktitle = {Proceedings of the Tenth Annual ACM Symposium on Theory of Computing},
pages = {216–226},
numpages = {11},
location = {San Diego, California, USA},
series = {STOC '78}
}

@article{bermond2013directed,
  title={Directed acyclic graphs with the unique dipath property},
  author={Bermond, Jean-Claude and Cosnard, Michel and P{\'e}rennes, St{\'e}phane},
  journal={Theoretical Computer Science},
  volume={504},
  pages={5--11},
  year={2013},
  publisher={Elsevier}
}

@article{araujo2012good,
  title={Good edge-labelling of graphs},
  author={Araújo, Júlio and Nathann Cohen and Frédéric Giroire and Frédéric Havet},
  journal={Discrete Applied Mathematics},
  volume={160},
  number={18},
  pages={2502--2513},
  year={2012},
  publisher={Elsevier}
}

@misc{sadhukhan2024shift,
      title={Shift Graphs, Chromatic Number and Acyclic One-Path Orientations}, 
      author={Arpan Sadhukhan},
      year={2024},
      eprint={2308.14010},
      archivePrefix={arXiv},
      primaryClass={math.CO}
}

@book{mukherjee1997optical,
  title={Optical communication networks},
  author={Mukherjee, Biswanath},
  year={1997},
  month={7},
  publisher={New York, NY: McGraw-Hill}
}
